\documentclass[12pt]{amsart}
\usepackage{amssymb,latexsym,amsmath}
\usepackage[mathscr]{eucal}

\numberwithin{equation}{section}

\def\be{\begin{equation}}
\def\ene{\end{equation}}

\def\BF{\mathbb F}

\def\BP{\mathbb P}

\def\fa{\mathfrak a}
\def\fb{\mathfrak b}\def\fgl{\mathfrak g\mathfrak l}

\def\tdim{{\rm dim}}

\def\hd{,...,}

\def\cF{{\mathcal F}}
\def\cG{{\mathcal G}}

\def\cO{{\mathcal O}}

\def\11{\mathbf 1}

\def\fh{{\mathfrak h}}

\def\fsl{{\mathfrak {sl}}}
\def\fsp{{\mathfrak {sp}}}

\def\fso{{\mathfrak {so}}}
\def\fe{{\mathfrak e}}\def\fm{{\mathfrak m}}

\def\ff{{\mathfrak f}}

\def\fg{{\mathfrak g}}

\def\fp{{\mathfrak p}}
\def\fk{{\mathfrak k}}

\def\a{\alpha}

\def\b{\beta}

\def\n{\nu}
\def\d{\delta}

\def\m{\mu}
\def\up#1{{}^{({#1})}}
\def\e{\varepsilon}
\def\ot{{\mathord{ \otimes } }}
\def\op{{\mathord{\,\oplus }\,}}

\def\fgl{\frak g\frak l}\def\fsl{\frak s\frak l}

\def\op{\oplus}
\def\BF{\Bbb F}

\def\ff#1{\Bbb F\Bbb F^{#1}}

\def\op{\oplus}

\def\ff{\mathfrak f}

\def\ul{\underline}

\def\a{\alpha}
\def\b{\beta}

\def\fso{\frak{so}}

\def\BP{\mathbb  P}

\def\fp{\mathfrak  p}

\def\fg{\mathfrak  g}

\def\hd{, \hdots ,}

\def\tdim{\operatorname{dim}}
\def\tker{\operatorname{ker}}

\def\tmod{\operatorname{mod}}
\def\tmin{\operatorname{min}}

\def\up#1{{}^{({#1})}}

\def\tspan{\operatorname{span}}

\def\be{\begin{equation}}
\def\ene{\end{equation}}

\def\tspan{{\rm span}}

\def\tspan{{\rm span}}

\def\a{\alpha}
\def\b{\beta}
\def\c{\gamma}
\def\d{\delta}

\def\e{\varepsilon}

\def\z{\zeta}
\def\m{\mu}
\def\n{\nu}

\def\w{\omega}

\def\x{\xi}

\def\bC{\mathbb C}

\def\bP{\mathbb P}
\def\bQ{\mathbb Q}

\def\bZ{\mathbb Z}

\def\cF{\mathcal F}
\def\cG{\mathcal G}

\def\cK{\mathcal K}

\def\cO{\mathcal O}

\def\fa{\mathfrak{a}}
\def\fb{\mathfrak{b}}

\def\fe{\mathfrak{e}}
\def\ff{\mathfrak{f}}

\def\fg{{\mathfrak{g}}}

\def\fgl{\mathfrak{gl}}
\def\fh{\mathfrak{h}}

\def\fk{\mathfrak{k}}

\def\fm{\mathfrak{m}}

\def\fp{\mathfrak{p}}

\def\fsl{\mathfrak{sl}}
\def\fso{\mathfrak{so}}
\def\fsp{\mathfrak{sp}}

\def\sfa{\mathsf{a}}
\def\tta{\mathtt{a}}

\def\tAd{\mathrm{Ad}}

\def\tAnn{\mathrm{Ann}}

\def\td{\mathrm{d}}

\def\tdim{\mathrm{dim}}

\def\ttE{\mathtt{E}}
\def\tEnd{\mathrm{End}}

\def\be{\mathbf{e}}

\def\sff{\mathsf{f}}

\def\sfF{\mathsf{F}}

\def\tGL{\mathrm{GL}}

\def\tti{\mathtt{i}}

\def\tId{\mathrm{Id}}

\def\ttj{\mathtt{j}}

\def\ttJ{\mathtt{J}}

\def\tker{\mathrm{ker}}
\def\sfL{\mathsf{L}}

\def\tmin{\mathrm{min}}
\def\tmod{\mathrm{mod}}

\def\sfN{\mathsf{N}}

\def\sfq{\mathsf{q}}
\def\ttq{\mathtt{q}}

\def\sfr{\mathsf{r}}

\def\tspan{\mathrm{span}}
\def\tSeg{\mathrm{Seg}}

\def\tSym{\mathrm{Sym}}

\def\sfT{\mathsf{T}}

\def\bu{\mathbf{u}}

\def\sfv{\mathsf{v}}

\def\bx{\mathbf{x}}

\def\op{\oplus}
\def\ot{\otimes}
\def\ul{\underline}
\def\wt{\widetilde}
\def\wh{\widehat}

\def\lefthook{\hbox{\small{$\lrcorner\, $}}}

\def\half{\tfrac{1}{2}}

\def\inj{\hookrightarrow}

\newcounter{numcnt}

\newcounter{cnt}

\newcounter{acnt}

\newenvironment{a_list}{ 
  \begin{list}{{\emph{(\alph{acnt})}}}
   {\usecounter{acnt} \setlength{\itemsep}{3pt}
    \setlength{\leftmargin}{25pt} \setlength{\labelwidth}{20pt} }
   }
   {\end{list}}
\newcounter{Acnt}

\newcounter{icnt}

\newcounter{Icnt}

\newcounter{exam_cnt}

\newcounter{mccnt}
\newenvironment{circlist}{ 
  \begin{list}{$\circ$}
   {\usecounter{cnt} \setlength{\itemsep}{2pt}
    \setlength{\leftmargin}{15pt} \setlength{\labelwidth}{20pt} }
   }
   {\end{list}}

\newtheorem{corollary}[equation]{Corollary}
\newtheorem*{corollary*}{Corollary}
\newtheorem{lemma}[equation]{Lemma}
\newtheorem*{lemma*}{Lemma}
\newtheorem{proposition}[equation]{Proposition}
\newtheorem*{proposition*}{Proposition}
\newtheorem{theorem}[equation]{Theorem}
\newtheorem*{theorem*}{Theorem}

\theoremstyle{remark}
\newtheorem*{assume*}{Assume}

\newtheorem*{claim*}{Claim}

\newtheorem{definition}[equation]{Definition}
\newtheorem*{definition*}{Definition}

\newtheorem*{example*}{Example}
\newtheorem*{hint*}{Hint}
\newtheorem*{notation*}{Notation}
\newtheorem*{answer*}{Answer}
\newtheorem{remark}[equation]{Remark}
\newtheorem*{remark*}{Remark}
\newtheorem*{remarks*}{Remarks}

\numberwithin{HWeq}{section}
\theoremstyle{definition}

\begin{document}
\title[Fubini-Griffiths-Harris rigidity]{Fubini-Griffiths-Harris rigidity of homogeneous varieties}
\author{J.M. Landsberg \& C. Robles}
\thanks{Both authors supported by NSF-DMS 1006353.}
\email{jml@math.tamu.edu, robles@math.tamu.edu}
\address{Mathematics Department, TAMU Mail-stop 3368, College Station, TX  77843-3368}
\subjclass[2010]{53A55, 
 53C24, 
 53C30, 
 14M17} 
\date{\today}
\begin{abstract}
Upper bounds on projective rigidity of each homogeneously embedded homogeneous variety are determined; and a new, invariant characterization of the Fubini forms is given.
\end{abstract}

\maketitle

\section{Introduction}
 
\subsection{History}

The study of the projective rigidity of homogeneous varieties dates back to Monge; he proved that conics in the plane are characterized by a fifth-order ODE. Fubini \cite{fubini} then showed that the higher dimensional smooth quadric hypersurfaces are characterized by a third-order system of PDE; we say these quadric hypersurfaces are {\it (Fubini-Griffiths-Harris) rigid} at order three (cf. Definition \ref{D:FGH}).  Later, Griffiths and Harris \cite{MR559347} conjectured that the Segre variety $\tSeg(\bP^2\times \bP^2)\subset \bP^8$ could be characterized by a {\it second-order} system of PDE; that is, the Segre variety is rigid at order two.  This was shown to be the case in \cite{MR1713310}; indeed, any irreducible, rank two, compact Hermitian symmetric space (CHSS) in its minimal homogeneous embedding (except for the quadric hypersurface) is rigid at order two \cite{MR2269783}.  Hwang and Yamaguchi \cite{MR2030098} then solved the rigidity problem for all homogeneously embedded irreducible CHSS: excluding $\bP^m$ and the quadric hypersurface, any irreducible, homogeneously embedded compact Hermitian symmetric space with osculating sequence of length $f$ is rigid at order $f$. (The length of the osculating filtration associated to a CHSS in its minimal homogeneous embedding is equal to its rank.)
 
Prior to \cite{MR2030098} the rigidity arguments all used standard exterior differential systems (EDS) machinery.  In \cite{MR2030098} Hwang and Yamaguchi observed that (partial) vanishing of a Lie algebra cohomology group implies rigidity.  Results of Kostant \cite{MR0142696} reduce the computation of the Lie algebra cohomology to Weyl group combinatorics.  Because Lie algebra cohomology appears in the analysis of the EDS  associated to the rigidity problem only in the case that $G/P$ is CHSS, it seemed that the approach of Hwang and Yamaguchi did not extend to the general case.  

\subsection{Extending Hwang and Yamaguchi}

This paper is a sequel to \cite{LRrigid}.  The two articles complete a project to extend the approach of Hwang and Yamaguchi to general $G/P \inj \bP V$.

The geometry and representation theory of $G/P\inj \bP  V$ each determine filtrations of $V$.  The filtrations coincide only when $G/P$ is CHSS.  Both filtrations give rise to families of exterior differential systems: the Fubini systems, which determine the order of rigidity, are induced by the geometric filtration; the Lie algebra graded systems, introduced in \cite{LRrigid}, are determined by the representation theoretic filtration.  The two EDS agree precisely when the filtrations agree: only when $G/P$ is CHSS. 

In \cite{LRrigid}, we took the representation-theoretic filtration as our starting point.  However, it still was not clear how to apply Kostant's machinery as the differential system associated to the filtration does not lead to Lie algebra cohomology.  This second difficulty was overcome by introducing a new type of exterior differential system, which we called a {\it filtered EDS}.  These systems can be defined on any manifold $\Sigma$ equipped with a filtration of the cotangent bundle, cf. Definition \ref{D:filtsys}.  Analysis of the filtered EDS does lead to Lie algebra cohomology, and we applied Kostant's results to study the rigidity of filtered systems.

The filtered systems are easily related to the system associated to the representation theoretic filtration, but the connection with the original Fubini systems was unclear in general.  (In any given case, one could determine the relation by brute computation, and we did so for adjoint varieties in \cite{LRrigid}.)  In this paper we introduce the \emph{graded Fubini system}, which is a classical EDS.  The grading is induced by the representation theory, and is akin the `weighting' of variables in CR--geometry; see, for example, \cite[pp. 229]{MR0425155}.  Integral manifolds of the graded Fubini system are integral manifolds of a filtered system.  The rigidity results of \cite{LRrigid} then yield our main result (Theorem \ref{T:main}).

\subsection{Rigidity of homogeneous varieties}

Our main result (Theorem \ref{T:main}) is a precise rigidity statement in terms of the graded Fubini system.  It implies the following general Fubini rigidity statement for arbitrary $G/P\hookrightarrow \bP V$. 

\begin{theorem}[Corollary to Theorem \ref{T:main}]\label{bbmainthm}  Let $G$ be a complex semi-simple Lie group with irreducible representation $V$ of highest weight $\pi$.  Let $G/P \hookrightarrow \bP V$ be a homogeneous embedding, and $\ttE$ the grading element associated to $P$.  Set $\ttq = \pi(\ttE) + \pi^*(\ttE)$, with $\pi^*$ the highest weight of $V^*$.
\begin{a_list}
\item If no factor of $G/P$ corresponds to a quadric hypersurface or $A_r/P_\ttJ$, with $1$ or $r$ in $\ttJ$, then $G/P \inj \bP V$ is Fubini-Griffiths-Harris rigid 
at   order $\ttq$.
\item If no factor of $G/P$ corresponds to $A_r/P_\ttJ$, with $1$ or $r$ in $\ttJ$, then $G/P\inj\bP V$ is Fubini-Griffiths-Harris rigid at   order $\ttq+1$.
\end{a_list}
\end{theorem}

\begin{remark*}
Fubini-Griffiths-Harris rigidity is defined in Definition \ref{D:FGH}, and the grading element is defined in Section \ref{S:grel}.
\end{remark*}

\begin{example*}[Maximal parabolic]
The values of $\ttq$ in the case that $\pi = \pi_\tti$ is a fundamental weight are listed in Table \ref{t:max_par}.  (Here $\ttJ = \{\tti\}$.)
\begin{table}[h] 
\renewcommand{\arraystretch}{1.3}
\caption{Values of $\ttq$ for $G/P_\tti \subset \bP V_{\pi_\tti}$.}
\begin{tabular}{|c||c|c|c|c|c|c|c|}
\hline
 $G$ & $A_r$ & \multicolumn{2}{|c|}{$B_r$} &
 \multicolumn{2}{|c|}{$C_r$} & \multicolumn{2}{|c|}{$D_r$} \\
\hline
 $\tti$ & $1 \le \tti \le r$ & $\tti < r $ & $r$ & $\tti < r$ & $r$ 
     & $\tti<r-1$ & $r-1$, $r$ \\
\hline
 $\ttq$ & $\tmin\{\tti, r+1-\tti\}$ & $2\tti$ & $r$ & $2\tti$ & $r$ 
     & $2\tti$ & $\lfloor r/2 \rfloor$ \\
\hline
\end{tabular}
\bigskip\\
\begin{tabular}{|c||c|c|c|c|c|c|c|c|c|c|c|}
\hline
$G$ & \multicolumn{4}{|c|}{$E_6$}   
    & \multicolumn{7}{|c|}{$E_7$}   \\  \hline
$\tti$ & $1,6$ & $2$ & $3,5$ & $4$
    & $1$ & $2$ & $3$ & $4$ & $5$ & $6$ & $7$  \\ \hline
$\ttq$ & $2$   & $4$ & $6$   & $12$
    & $4$ & $7$ & $12$ & $24$ & $15$ & $8$ & $3$ \\ \hline
\end{tabular} 
\bigskip\\
\begin{tabular}{|c||c|c|c|c|c|c|c|c|c|c|c|}
\hline
$G$ & \multicolumn{8}{|c|}{$E_8$} & \multicolumn{2}{|c|}{$F_4$} 
    & $G_2$ \\ \hline
$\tti$ & $1$ & $2$ & $3$ & $4$ & $5$ & $6$ & $7$ & $8$ 
    & $1, 4$ & $2, 3$ & $1, 2$ \\ \hline
$\ttq$ & $8$ & $16$ & $28$ & $60$ & $40$ & $24$ & $12$ & $4$
    & $4$ & $12$ & $4$ \\
\hline
\end{tabular}
\label{t:max_par}
\end{table}
\end{example*}

\begin{example*}[Flag varieties]
The subgroup $P$ is Borel (and $\ttJ = \{1,\ldots,r\}$) when $\pi$ is the sum $\rho = \pi_1 + \cdots + \pi_r$ of the fundamental weights, $r = \mathrm{rank}(G)$.  The corresponding values of $\ttq = 2\rho(\ttE)$ are listed in Table \ref{t:min_par}.
\begin{table}[h] 
\renewcommand{\arraystretch}{1.3}
\caption{Values of $\ttq$ for $G/B \subset \bP V_{\rho}$.}
\begin{tabular}{|c||c|c|c|c|}
\hline
$G$ & $A_r$ & $B_r$ & $C_r$ & $D_r$ \\
\hline
$\ttq$ & $r(r+1)(r+2)/6$ & $r(r+1)(4r-1)/6$ & $r(r+1)(4r-1)/6$ 
    & $r(r-1)(2r-1)/3$  \\
\hline
\end{tabular}
\bigskip\\
\begin{tabular}{|c||c|c|c|c|c|}
\hline
$G$ & $E_6$ & $E_7$ & $E_8$ & $F_4$ & $G_2$ \\
\hline
$\ttq$ & $156$ & $399$ & $1240$ & $110$ & $16$ \\
\hline
\end{tabular}
\label{t:min_par}
\end{table}
\end{example*}

In the case that $G/P$ is not CHSS, Theorem \ref{T:main} is stronger than Theorem \ref{bbmainthm}.  The grading element allows one to refine the osculating sequence, and we require only partial (graded) agreement of the differential invariants up to order $\ttq$ (or $\ttq+1$).   

\subsection{Differential geometry of \boldmath $G/P \hookrightarrow \bP V$ \unboldmath } \label{S:I4}

Our colleagues in algebraic geometry have often asked us for a representation-theoretic description of the Fubini forms for homogeneous varieties $G/P \hookrightarrow \bP V$.  Proposition \ref{invarfubgp}, which generalizes the invariant description \cite[Proposition 2.3]{MR1966752} of the fundamental forms to Fubini forms, provides this characterization.

\subsection*{Acknowledgements}
We thank the anonymous referees for their careful readings and ameliorating feedback.

\section{Fubini forms} \label{S:FubF}

Given $0 \not= v \in V$, let $[v] \in \bP V$ denote the corresponding line.  Given any set $S \subset \bP V$, let $\widehat S$ denote the corresponding cone in $V \backslash\{0\}$.   Fix the following index ranges
$$ \begin{array}{r @{\ \le \ } c @{\ \le \ } l}
  0 & i , j & n =\tdim \,V\,,\\
  1 & \a,\b & m =\tdim \,X \,,\\
  1+m & \mu,\nu & n \,.
\end{array}$$

\subsection{Fubini forms in frames}\label{S:Fub}

We review the computation of Fubini forms via moving  frames briefly here; see \cite[\S2]{MR2739065} or \cite{MR2474431} for details.  Section \ref{S:F5} presents another interpretation of the Fubini forms.

Let $X\subset \BP V$ be the general points of a projective variety.  Given $x \in X$, let $\wh T_x X \subset V$ be the linear space tangent to $\wh X$ at $y \in \hat x$.  (The definition of $\wh T_x X$ does not depend on our choice of $y$.)  Let 
$$
  T_xX \ = \ (\wh T_x X / L_x) \ot L_x^*
  \quad \hbox{ and } \quad 
  N_xX \ = \ ( V / \wh T_xX) \ot L_x^*
$$ 
respectively denote the tangent and normal spaces at  $x$.  To the filtration
\begin{equation} \label{E:1filt}
  L_x \ := \ \hat x \ \subset \ \wh T_x X \ \subset \ V
\end{equation}
we associate the bundle $\cF_X$ of first-order adapted frames over $X$:  this is the set of bases $v = (v_0,\ldots,v_n)$ of $V$ such that $v_0 \in L_x$, with $x \in X$ a general point, and 
\begin{equation} \label{E:cFX}
  \wh T_x X \ = \ \tspan\{ v_0,\ldots , v_m \} \, .
\end{equation}

Given a frame $v \in \cF_X$ over $x \in X$, the tangent space $T_xX$ is spanned by the $\ul{v}_\a := ( v_\a \ \tmod \ L_x ) \ot v^0$; let $\{\ul{v}^\a\}_{\a=1}^m$ denote the dual basis of the cotangent space, and $\ul{v}^{\a_1\cdots\a_k} \in \tSym^kT^*_xX$ the symmetric product $\ul{v}^{\a_1}\cdots\ul{v}^{\a_k}$.  Similarly, the normal spaces $N_xX$ is spanned by $\ul{v}_\m := ( v_\a \ \tmod \ \wh T_xX) \ot v^0$.  So the $k$--th Fubini form $F^k_v \in (N_xX) \ot (\tSym^k\,T^*_xX)$ at $v$ may be expressed as 
$$
  F^k_v \ = \ r^\m_{\a_1\cdots\a_k} \, \ul{v}_\m \ot \ul{v}^{\a_1\cdots\a_k} \, .
$$
The coefficients $r^\m_{\a_1\cdots\a_k} : \cF_X \to \bC$ are obtained as follows.

\begin{remark} \label{R:v}
Fixing a frame $\sfv \in \cF_X$ we may identify the Lie group $\tGL(V)$ with the set of all ordered-bases (frames) $v = (v_0,\ldots,v_n)$ of $V$.  The identification maps
$$
  g \in \tGL(V) \quad \mapsto \quad v = g \cdot \sfv \, .
$$ 
\end{remark}

Given $v = (v_0,\ldots,v_n) \in \tGL(V)$, let $v^* = (v^0,\ldots,v^n)$ denote the dual basis of $V^*$.  Let $\w\in\Omega^1(\tGL(V),\fgl(V))$ denote the $\fgl(V)$--valued Maurer-Cartan form on $\tGL(V)$.  Define 1--forms $\w^i_j \in \Omega^1(\tGL(V))$ at $v$ by 
$$
  \w_v \ =: \ \w^i_j\,v_i\ot v^j \, .
$$
The Maurer-Cartan form satisfies the Maurer-Cartan equation
\begin{equation} \label{E:mc}
  \td \w \ = \ -\half [ \w , \w ] \quad \hbox{ or } \quad
  \td \w^i_j \ = \ - \w^i_\ell \wedge \w^\ell_j \, .
\end{equation}

Equation \eqref{E:cFX} implies 
\begin{equation} \label{E:mc1}
  \w^\m_0 \ = \ 0 \quad\hbox{on}\quad \cF_X \, .
\end{equation}
An application of \eqref{E:mc} to \eqref{E:mc1} yields $0 = \w^\m_\a\wedge\w^\a_0$.  By  Cartan's Lemma \cite[Lemma A.1.9]{MR2003610} there exist functions $r^\m_{\a\b} = r^\m_{\b\a}$ on $\cF_X$ such that
\begin{equation} \label{E:F2} 
  \w^\m_\a \ = \ r^\m_{\a\b} \, \w^\b_0 \, .
\end{equation}
The functions $r^\m_{\a\b}$ are the coefficients of the second Fubini form.

Given two tensors $T_{\b_1 \ldots \b_k}$ and $U_{\b_{k+1} \ldots \b_{k+\ell}}$, let $S_{k+\ell}$ denote the symmetric group on $k+\ell$ letters.  Let
$$
  T_{(\b_1 \ldots \b_k} U_{\b_{k+1} \ldots \b_{k+\ell})} \ = \ 
  \frac{1}{(k+\ell)!} \sum_{\sigma\in S_{k+\ell}} \, 
  T_{\sigma(\b_1) \ldots \sigma(\b_k)} U_{\sigma(\b_{k+1}) \ldots \sigma(\b_{k+\ell})}
$$ 
denote the   symmetrization  of their product.  For example, $T_{(\b_1} U_{\b_2)} = \half( T_{\b_1} U_{\b_2} + T_{\b_2} U_{\b_1})$.  We exclude from the symmetrization operation any index that is outside the parentheses.  For example, in $r^\m_{\a(\b_1 \ldots \b_{p-1}} \, \w^\a_{\b_p)}$ we symmetrize over only the $\b_i$, excluding the $\a$ index.   

\begin{proposition}[\cite{MR2474431}]
\label{prop:F} Let $X\subset \bP  V$ be a complex submanifold and let $x\in X$.
Set $r^\m_\a = 0$.  Assume $k>1$.  The coefficients $r^\m_{\a_1\cdots\a_k\b}$ of $F^{k+1}$,   fully symmetric in their lower indices, are defined by 
\begin{eqnarray}
  \nonumber
  r^\m_{\a_1 \ldots \a_k \b} \, \w^\b_0 & = &
  - \, \td r^\m_{\a_1 \ldots \a_k} \, - \, (k-1) \, r^\m_{\a_1 \ldots \a_k} \, \w^0_0
  \, - \, r^\n_{\a_1 \ldots \a_k} \, \w^\m_\n \hfill \\
  \label{E:F}
  & &
  + \ k \,  
  \left\{ (k-2) \, r^\m_{(\a_1 \ldots \a_{k-1}} \w^0_{\a_k)} \,  
       + \, r^\m_{\b(\a_1 \ldots \a_{k-1}} \, \w^\b_{\a_k)}  \right\} 
  \\ \nonumber
  & & 
  - \ \sum_{j=1}^{k-2} \, {\textstyle{\binom{k}{j}}} \, 
      \Big\{ 
      \ r^\m_{\b(\a_1 \ldots \a_j} \, r^\n_{\a_{j+1} \ldots \a_k)} \, \w^\b_\n \ 
      + \ (j-1)\, r^\m_{(\a_1 \ldots \a_j} \, r^\n_{\a_{j+1} \ldots \a_k)} \, 
                 \w^0_\n \, \Big\} \, . 
\end{eqnarray}
when $k>1$.
\end{proposition}

\begin{example*}
The coefficients of the third Fubini form $F_3$ are given by
\begin{equation}
\label{E:F3}
  r^\m_{\a\b\c} \w^\c_0 \ = \ - \td r^\m_{\a\b} - r^\m_{\a\b} \, \w^0_0 
    - r^\n_{\a\b} \, \w^\m_\n + r^\m_{\a\e} \, \w^\e_\b + r^\m_{\b\e} \, \w^\e_\a \, . 
\end{equation}
\end{example*}

\subsection{The osculating filtration and the bundle \boldmath $\cF^\ell_X$ \unboldmath} \label{S:Fell}

The filtration \eqref{E:1filt} may be refined to the \emph{osculating filtration}
\begin{equation}\label{oscseq}
  L_x \ := \ \hat x \ \subset \ \wh T_xX \ \subset \ \wh T_x\up 2 X
  \ \subset \,\cdots\, \subset \ \wh T_x \up \ell X \ = \ V \, .
\end{equation}
By definition, the linear space $\wh T^{(k)}_xX$ is spanned by the derivatives $y^{(j)}(0)$, of order $j \le k$, of smooth curves $y(t)$ in $\wh X$ with $y(0) \in \hat x$.  Let $\cF^\ell_X \subset \cF_X$ be the sub-bundle of frames adapted to the filtration \eqref{oscseq}.  We assume, without loss of generality that the fixed frame $\sfv$ of Remark \ref{R:v} lies in $\cF^\ell_X$.

\subsection{On \boldmath $G/P \inj \bP V$ \unboldmath} \label{S:Fub_homog}

Let $G$ be a complex semi-simple Lie group with Lie algebra $\fg$.  
\begin{center}
\emph{Fix a choice of Borel subalgebra $\fb \subset \fg$ with Cartan subalgebra $\fh \subset \fb$.}
\end{center}
\noindent Let $\{ \pi_1,\ldots,\pi_r\} \subset \fh^*$ be the fundamental weights of $\fg$.  The parabolic subalgebras (containing $\fb$) are in bijection with subsets $\ttJ \subset \{ 1 , \ldots, r \}$.  The maximal parabolics are given by $|\ttJ| = 1$, the choice of a single fundamental weight; the minimal parabolic $\fb$ is given by $\ttJ = \{1,\ldots, r\}$.

Let $V$ be an irreducible $G$-module of highest weight $\pi \in \fh^*$, and let $0\not=\sfv_0\in V$ be a highest weight vector.  The rational homogeneous variety $G\cdot[\sfv_0] = Z \subset \bP V$ is the unique closed $G$--orbit in $\bP V$.  The stabilizer of the highest weight line
$$
  o \ := \ [\sfv_0] \ \in \ Z
$$
is a parabolic subgroup $P \subset G$, and $Z \simeq G/P$.  The subset $\ttJ$ associated to $P$ is defined by $\pi = \sum_{\ttj\in\ttJ} p^\ttj\,\pi_\ttj$, with $p^\ttj > 0$.  The variety $Z$ is the \emph{homogeneous embedding of $G/P$ into $\bP V$}.  

Given a fixed frame $\sfv\in\cF_Z$, there is a natural embedding of $G$ into $\cF_Z$ as
$$
  \cG \ := \ G \cdot \sfv \, \subset \cF_Z \,.
$$
Note that $T_oZ \simeq \fg/\fp$ as a $\fp$--module.  More generally, if 
$$
  z \ = \ g \cdot o \,,
$$
then $T_zZ \simeq \fg/\tAd_g(\fp)$ as an $\tAd_g(\fp)$--module.

The parabolic algebra $\fp$, and our choice of Cartan and Borel subalgebras $\fh \subset \fb \subset \fp$,  determines a decomposition 
\begin{equation}  \label{E:coarse}
  \fg \ = \ \fg_+ \,\op\, \fg_0 \,\op\, \fg_- \quad \hbox{ with } \quad
  \fp \ = \ \fg_{\ge0} \,;
\end{equation}
where $\fg_0$ is a reductive subalgebra, and $\fg_\pm$ are nilpotent subalgebras.  (See also \S\ref{S:grel}.)  Moreover, $\fg_\pm$ are $\fg_0$--modules, and $\fg_-{}^* \simeq \fg_+$ (under the Killing form).  As a $\fg_0$--module, $T_oZ \simeq \fg_-$.  More generally,
\begin{equation}\label{E:TgZ}
  T_z Z \ \simeq \ \tAd_g(\fg_-) \quad\hbox{ and }\quad
  T^*_z Z \ \simeq \ \tAd_g(\fg_+) 
\end{equation}
as $\tAd_g(\fg_0)$--modules.
In particular, 
\begin{equation} \label{E:SymTgZ}
  \tSym^k T_z Z \ \simeq \ \tSym^k \tAd_g(\fg_-) 
  \ \stackrel{(\ast)}{\simeq} \ U^k(\tAd_g(\fg_-)) \ \simeq \ 
  \tAd_g \,U^k(\fg_-) \ \subset \ U^k(\fg) 
\end{equation}
as $\tAd_g(\fg_0)$--modules; the identification $(\ast)$ is by the Poincar\'e-Birkoff-Witt Theorem.  We denote the Lie algebra action on a $G$-module by a lower dot: given $\z\in \fg$ and $w\in V$, we write $\z.w$.  

The action of $\fg$ on $V$ induces actions of the universal enveloping algebra $U(\fg)$ on $V$, its dual $V^*$ and $\fgl(V) = \tEnd(V) = V \ot V^*$.  The action on $\fgl(V)$ restricts to the adjoint action on $\fg \subset \fgl(V)$.  Given $\bu \in U(\fg)$, define a 1-form $\bu.\w \in \Omega^1(\cG,\fg)$ by 
$$
  (\bu.\w)(\xi) \ := \ \bu.(\w(\xi)) \quad \hbox{for all } \xi \in T\cG \, .
$$ 
Alternatively,
$$
  (\bu.\w)_v \ := \ \w^i_j \, \bu.(v_i\ot v^j) \ = \ 
  \w^i_j\, \left( (\bu.v_i) \ot v^j \,+\, v_i \ot(\bu.v^j) \right) \, .
$$
Define 1--forms $(\bu.\w)^i_j \in \Omega^1(\cG)$ at $v\in\cG$ by 
$$
  (\bu.\w)_v \ =: \ (\bu.\w)^i_j \,v_i \ot v^j \, .
$$

Every $A^i_j \,v_i\ot v^j \in \fg \subset \fgl(V)$ satisfies $A^\m_0 = 0$.  In particular, 
\begin{equation} \label{E:u.w}
  (\bu.\w)^\m_0 \ = \ 0 \, .
\end{equation}
The fact that $\tAd_g(\fp)$ preserves the line through $v_0 = g\cdot \sfv_0$ implies that \eqref{E:u.w} is trivial when $\bu \in U(\tAd_g(\fp))$.  When $\bu \in U(\tAd_g(\fg_-))$, Lemma \ref{L:key} asserts that \eqref{E:u.w} is the equations \eqref{E:F2} and \eqref{E:F} defining the coefficients of the Fubini form at $v = g\cdot \sfv \in \cG$.

\begin{lemma}[Key observation] \label{L:key}
Let $V$ be an irreducible module for the reductive group $G$.  Given $g \in G$, let $z = g \cdot o \in Z \simeq G/P$ and $v = g\cdot\sfv \in \cG$.  The equations \eqref{E:F2} and \eqref{E:F} defining the coefficients of the $(k+1)$-st order Fubini form of $Z \simeq G/P$ at $v$ are $0 = (\bu.\w)^\mu_0$ with $\bu \in U^k(\tAd_g(\fg_-)) \simeq \tSym^kT_zZ$.
\end{lemma}

\begin{remark*}
The expression \eqref{E:F} defining the coefficients of the Fubini forms simplifies substantially on $\cG$; see \eqref{E:F_homog}.
\end{remark*}

\noindent Note that the identification $U^k(\tAd_g(\fg_-)) \simeq \tSym^kT_zZ$ is \eqref{E:SymTgZ}.  The lemma is proven below.  First we illustrate the computation at $\sfv \in \cG$.  Refine the indices $\{\mu\} = \{\mu_2,\ldots,\mu_\ell\}$
to respect the filtration \eqref{oscseq}.  On $\cF^\ell_Z$
$$
  \w^{\mu_k}_\a \ = \ 0 \quad \forall \ k > 2 \,, \quad \hbox{ and } \quad
  \w^{\mu_k}_{\nu_j} \ = \ 0 \quad \forall \ k-j > 1 \, . 
$$

Fix a basis $u_1,\ldots,u_m$ of $\fg_-\subset\fgl(V)$.  Without loss of generality, we may suppose that $\sfv$ was selected so that $\sfv_\a = u_\a . \sfv_0$, and $\sfv \in \cF_Z^\ell$.  The latter implies 
$$
  \cG \ \subset \ \cF_Z^\ell \, .
$$

\begin{example*}
Fix $A = u_\a \,, \ B = u_\b \,,\ C = u_\c \in \fg_-$.  
Using the inclusion $\fg_-\subset \fgl(V)$,
write  $A^i_j$ by $A = A^i_j \, \sfv_i \ot \sfv^j \in \fgl(V)$.  Note that
$$
  A^j_0 = \d^j_\a \quad \hbox{ and } \quad 
  (AB)^{\mu_2}_0 \,=\,A^{\mu_2}_\b \ = \ B^{\mu_2}_\a \,=\, (BA)^{\mu_2}_0 \, .
$$
  Write $\w = \w^i_j\,\sfv_i \ot \sfv^j$.  Then
$$
  (A.\w)^{\mu_2}_0 \ = \ \w^\e_0 \, A^{\mu_2}_\e \ - \ \w^{\mu_2}_\a \,. 
$$
Thus $(A.\w)^{\mu_2}_0 = 0$ is \eqref{E:F2} and 
\begin{equation} \label{E:r2}
  r^{\mu_2}_{\a\b} \ = \ A^{\mu_2}_\b \ = \ B^{\mu_2}_\a
  \ = \ (AB)^{\mu_2}_0 \ = \ (BA)^{\mu_2}_0 \, .
\end{equation}
Next,
\begin{eqnarray*}
  (AB.\w)^{\mu_3}_0 & = & r^{\nu_2}_{\a\b}\,\w^{\mu_3}_{\nu_2} \ - \ 
  A^{\mu_3}_{\nu_2}\,r^{\nu_2}_{\b\e}\,\w^\e_0 
\\ 
 (AB.\w)^{\mu_2}_0 & = &   
  r^{\mu_2}_{\a\b}\,\w^0_0 \ + \ r^{\nu_2}_{\a\b}\,\w^{\mu_2}_{\nu_2} \ - \ 
  r^{\mu_2}_{\a\e}\,\w^\e_\b \ - \ r^{\mu_2}_{\e\b}\,\w^\e_\a \\
  & & + \left(
    B^{\mu_2}_\e\,A^\e_\c \,+\, A^\e_\b\,r^{\mu_2}_{\c\e} \,-\, 
    A^{\mu_2}_{\nu_2}\,r^{\nu_2}_{\b\c} \right. ) \w^\c_0 
\end{eqnarray*}
are the equations of \eqref{E:F3}, and 
\begin{eqnarray*}
  r^{\mu_3}_{\a\b\c} & = & 
  A^{\mu_3}_{\nu_2}\,r^{\nu_2}_{\b\c} \ = \ 
  A^{\mu_3}_j \,B^j_k \,C^k_0 \ = \ (ABC)^{\mu_3}_0 \, . \\
  r^{\mu_2}_{\a\b\c} & = & 
    B^{\mu_2}_\e\,A^\e_\c \,+\, A^\e_\b\,r^{\mu_2}_{\c\e} \,-\, 
    A^{\mu_2}_{\nu_2}\,r^{\nu_2}_{\b\c} \, .  
\end{eqnarray*}
\end{example*}

\begin{proof}[Proof of Lemma \ref{L:key}]
The equations \eqref{E:F} are obtained inductively.  The proof of the lemma is based on this iterated construction, so we briefly review the procedure; details may be found in \cite{MR2474431}.  Let $\varphi^\m_{\a_1\cdots\a_k}$ denote the right-hand side of \eqref{E:F}.  Set 
\begin{subequations}\label{SE:phipsi}
\begin{equation}
  \psi^\m_{\a_1\cdots\a_k} \ := \ 
  \varphi^\m_{\a_1\cdots\a_k} \, -\, r^\m_{\a_1\cdots\a_k\b}\,\w^\b_0 \, .
\end{equation}
Then \eqref{E:F} is equivalent to 
\begin{equation}
  0 \ = \ \psi^\m_{\a_1\cdots\a_k} \, .  
\end{equation}
\end{subequations}
Differentiating this expression one calculates 
$$
  0 \ = \ \td \psi^\m_{\a_1\cdots\a_k} \ = \ 
  \varphi^\m_{\a_1\cdots\a_k\b}\wedge\w^\b_0 \, .
$$

We now prove the lemma.  Recall that $\{u_\a\}_{\a=1}^m$ form a basis of $\fg_- \subset \fgl(V)$ satisfying $u_\a . \sfv_0 = \sfv_\a$.  Therefore, $\{\tAd_g(u_\a)\}_{\a=1}^m$ form a basis of $\tAd_g(\fg_-)$ satisfying $\tAd_g(u_\a) . v_0 = v_\a$.  When restricted to $\cG$ the Maurer-Cartan form takes values in $\fg$ and forms a coframing.  So we may define vector fields $\tilde u_\a$ on $\cG$ by $\w_{g\cdot\sfv}(\tilde u_\a) = \tAd_g(u_\a)$.  Given $\bu = \bu_g = \tAd_g(u_{\a_1}\cdots u_{\a_k})\in\tAd_gU^k(\fg_-)$, it suffices to show 
\begin{equation} \label{E:suf}
  \psi^\m_{\a_1\cdots\a_k} \ = \ -(\bu.\w)^\m_0 \, .
\end{equation}  

We first show that \eqref{E:suf} holds for $k=1$.  For this case we take $\bu = \tAd_g(u_\a)$.  We have $\psi^\m_\a = \w^\m_\a - r^\m_{\a\b} \w^\b_0$ and $\varphi^\m_\a = \w^\m_\a$.  Thus $\varphi^\m_\a\wedge\w^\a_0 = -\td \w^\m_0 = \half [\w,\w]^\m_0$.  Applying $\tilde u_\a \lefthook$ yields $-\psi^\m_\a = [\tAd_g(u_\a),\w]^\m_0 = (\tAd_g(u_\a).\w)^\m_0$.  This establishes \eqref{E:suf} for $k=1$.

For the general case note that the definition of $\varphi^\m_{\a_1\cdots\a_k}$ implies $\tilde u_\b \lefthook \varphi^\m_{\a_1\cdots\a_k} = r^\m_{\a_1\cdots\a_k\b}$.  Therefore, $\psi^\m_{\a_1\cdots\a_k\b} = -\tilde u_\b \lefthook \td \psi^\m_{\a_1\cdots\a_k}$.  If \eqref{E:suf} holds, then \eqref{E:mc} yields
\begin{eqnarray*}
   \psi^\m_{\a_1\cdots\a_k\b} & = & 
   \tilde u_\b \lefthook \td (\bu.\w)^\m_0 \ = \ 
   -\half \tilde u_\b \lefthook ( \bu.[\w,\w])^\m_0 \ = \ 
   - \left( \bu.[\tilde u_\b \lefthook\w,\w]\right)^\m_0 \\
   & = & - (\bu.(\tAd_g(u_\b).\w))^\m_0  \ = \
   -(\tAd_g(u_{\a_1}\cdots u_{\a_k} u_\b).\w)^\m_0 \, ,
\end{eqnarray*}
establishing \eqref{E:suf} for $k+1$.
\end{proof}

\begin{corollary} \label{C:cnst}
The coefficients of the Fubini forms are constant on $\cG \subset \cF_Z$.
\end{corollary}

\begin{proof}
Given $g \in G$, let $L_g : \cG \to \cG$ denote the map $v \mapsto g\cdot v$.  The left-invariance of the Maurer-Cartan form is the statement $(L_g)^*\w = \w$.  Similarly, the map $\bu : \cG \to U(\fg)$ given by $g\cdot \sfv \mapsto \bu_g = \tAd_g(u_{\a_1}\cdots u_{\a_k})$ transforms as $(L_g)^*\bu = \bu$.  It follows from \eqref{E:suf} that $(L_g)^*\psi^\m_{\a_1\cdots\a_k} = \psi^\m_{\a_1\cdots\a_k}$.  

To see that the coefficients of $F^2$ are constant, set $k=1$.  Equation \eqref{E:F2} and the left-invariance of $\w$ imply that the $r^\m_{\a\b}$ are constant.  The corollary now follows from \eqref{E:F} by induction on $k$.
\end{proof}

\begin{definition} \label{D:FGH}
Let $X \subset \bP V$ be a variety.  We say \emph{$X$ agrees with $G/P \inj \bP V$ to order $k$} if there exists a sub-bundle $\wt\cF_{X} \subset \cF_X$, defined over the general points of $X$, on which the coefficients $\tilde r$ of the Fubini forms $F^j$  satisfy $\tilde r = r_{|\cG}$ for all $j \le k$.

We say $G/P \inj \bP V$ is \emph{Fubini-Griffiths-Harris (FGH) rigid at order $k$} if agreement to order $k$ implies that $X$ is projectively equivalent to $G/P$.
\end{definition}

\section{The graded Fubini system of $G/P \inj \bP V$}

\subsection{The grading element} \label{S:grel}
With our choices of Borel and Cartan subalgebras $\fb \supset\fh$, a parabolic subalgebra $\fp = \fp(\ttJ)$ determines a grading element $\ttE = \ttE(\fp) \in \fh$ as follows.  Let $\{ E_1, \ldots , E_r\}$ be the basis of $\fh$ that is dual to the simple roots.  Define 
$$
  \ttE \ = \ \ttE(\fp) \ := \ \sum_{\ttj \in \ttJ} E_\ttj \, .
$$
Every $\fg$--module $M$ decomposes into a direct sum $M = \op\,M_d$ of $\ttE$--eigenspaces of eigenvalue $d \in \bQ$.  Call this direct sum the {\it $\ttE$--graded decomposition of $M$}, and $M_d$  \emph{the component of $M$ of $\ttE$--graded degree $d$}.  In the case that $V$ is an irreducible $\fg$--module of highest weight $\pi\in\fh^*$,  the decomposition is $V = V_{\pi(\ttE)} \op V_{\pi(\ttE)-1} \op V_{\pi(\ttE)-2} \op \cdots \op V_{-\pi^*(\ttE)}$, with $\pi^*$ the highest weight of the dual module $V^*$.  In general, $0 < \pi(\ttE) \in \bQ$.

The $\ttE$--graded decomposition of $\fg$
\begin{equation} \label{E:Zg}
  \fg \ = \ \fg_{\sfa} \,\op\,\cdots\,\op\,\fg_1 \, \op \, \fg_0 \,\op\, \fg_{-1} \,\op\,\cdots\,\op\,\fg_{-\sfa} \,,
\end{equation} 
satisfies $0 < \sfa \in \bZ$, and is compatible with \eqref{E:coarse}.  That is, the component of $\ttE$--graded degree 0 is the reductive Lie subalgebra $\fg_0$ of \eqref{E:coarse}, $\fg_+ = \op_{d>0} \fg_d$ and $\fg_- = \op_{d>0} \fg_{-d}$.

\subsection{Graded Fubini forms} \label{S:Zgr}

In this section we consider the Fubini forms of $G/P \inj \bP V$ at $\sfv\in\cG$.  The choice of Cartan and Borel subalgebras $\fh \subset \fb \subset \fg$ determines a splitting 
\begin{equation} \label{E:split0}
  V \ = \ \sfL  \,\op\, \sfT  \, \op\,  \sfN 
\end{equation}
of the filtration $\sfL :=  L_o  \,\subset\, \wh T_{[\sfv_0]}Z \,\subset\, V$.   There exist natural identifications
\begin{equation} \label{E:sfId}
  T_oZ \ = \ \sfT \ot \sfL^* \ = \ \fg_- \,,\quad
  N_oZ \ = \ \sfN \ot \sfL^* \,.
\end{equation}
 As a $\fg_0$--module,  $\sfN$ admits a decomposition into $\ttE$--eigenspaces $\sfN = \sfN_{\pi(\ttE)-2} \op \cdots \op \sfN_{-\pi^*(\ttE)}$,  where $\pi$ is the highest weight of $V$ and $-\pi^*$ is the lowest.  Note that $\sfL$ is an eigenline with eigenvalue $\pi(\ttE)$.  So 
$$
  \sfN \ot \sfL^* \ = \ (\sfN\ot \sfL^*)_{-2} \ \op \ \cdots \ \op \ 
  (\sfN\ot \sfL^*)_{-\ttq}
$$
where $\ttq = (\pi+\pi^*)(\ttE)$.  Similarly, the symmetric algebra admits a $\ttE$--graded decomposition
$$
   \tSym^\bullet \fg_+ \ = \ \bigoplus_{d\ge0} \tSym_d \, \fg_+ .
$$
  For example, $\tSym_0 \fg_+ = \bC$, $\tSym_1 \fg_+ = \fg_1$, 
$$
  \tSym_2 \fg_+ \,=\, \fg_2 \,\op\,\tSym^2\fg_1 \,,\quad
  \tSym_3 \fg_+ \,=\, \fg_3 \,\op\, (\fg_1 \op \fg_2) \,\op\, \tSym^3\fg_1 \, ,
$$
cf. \eqref{E:Zg}. Note that we index the $\ttE$--degree by a subscript and the 
polynomial degree by a superscript. 

\subsection{The graded Fubini system}

Corollary to Lemma \ref{L:key} we have

\begin{corollary} \label{C:deg}
The component of $F_\sfv$ in $(\sfN\ot\sfL^*)_{-s} \ot \tSym_d \, \fg_+$ vanishes if $s\not=d$.  Equivalently, $F_\sfv(\tSym_d\fg_+) \subset (\sfN\ot \sfL^*)_{-d}$.
\end{corollary}

\begin{definition}\label{D:Fd}
Let $\sfF_d$ denote the component of the Fubini form $F_\sfv$ taking value in $(\sfN \ot \sfL^*)_{-d} \ot \tSym_d \, \fg_+$.    We call $\sfF_d$ the \emph{graded Fubini form of $G/P\inj\bP V$ of graded-degree $d$}. 
\end{definition}

\begin{definition} \label{D:agree}
Let $X \subset \bP V$ be a variety.  We say \emph{$X$ agrees with $G/P \inj \bP V$ to graded-order $d$} if there exists a sub-bundle $\wt\cF_{X} \subset \cF_X$, defined over the general points of $X$, on which the coefficients $\tilde r$ of the Fubini forms satisfy $\tilde r^\m_{\a_1\cdots\a_k} = \sfr^\m_{\a_1\cdots\a_k}$ for all coefficients $\sfr^\m_{\a_1\cdots\a_k}$ of the graded Fubini forms $\sfF_c$ of $G/P \inj \bP V$ with $c \le d$ (cf. Definition \ref{D:Fd}).

We say $G/P \inj \bP V$ is \emph{graded-Fubini-Griffiths-Harris (graded-FGH) rigid at graded-order $d$} if agreement to graded-order $d$ implies that $X$ is projectively equivalent to $G/P$.
\end{definition}

\begin{theorem} \label{T:main}
Let $G$ be a complex semi-simple Lie group with irreducible representation $V$ of highest weight $\pi$.  Let $G/P \hookrightarrow \bP V$ be a homogeneous embedding, and $\ttE$ the grading element associated to $P$.  Set $\ttq = \pi(\ttE) + \pi^*(\ttE)$, with $\pi^*$ the highest weight of $V^*$.
\begin{a_list}
\item If no factor of $G/P$ corresponds to a quadric hypersurface or $A_r/P_\ttJ$, with $1$ or $r$ in $\ttJ$, then $G/P \inj \bP V$ is graded-Fubini-Griffiths-Harris rigid at graded-order $\ttq$.
\item If no factor of $G/P$ corresponds to $A_r/P_\ttJ$, with $1$ or $r$ in $\ttJ$, then $G/P\inj\bP V$ is graded-Fubini-Griffiths-Harris rigid at graded-order $\ttq+1$.
\end{a_list}
\end{theorem}
 
\noindent Theorem \ref{T:main} is proved in Section \ref{S:prf}.

\begin{example*}[Adjoint varieties]
In the case that $G$ is simple and $G/P \subset \bP\fg=\bP V_\pi$ is an adjoint variety the highest weight $\pi$ is the highest root.  Let $\{\pi_1 , \ldots , \pi_r\}$ denote the fundamental weights of $\fg$.  Then
\begin{center} \begin{tabular}{c|cccccccc}
  $\fg$ & $\fsl_{r+1}$ & $\fso$ & $\fsp$ & $\fe_6$ & $\fe_7$ & $\fe_8$
      & $\ff_4$ & $\fg_2$ \\ \hline 
  $\pi$ & $\pi_1+\pi_r$ & $\pi_2$ & $2\pi_1$ & $\pi_2$ & $\pi_1$ & $\pi_8$ 
      & $\pi_1$ & $\pi_2$
\end{tabular}\end{center}
In each case $\ttq = 4$.  If $\fg \not= \fsl$, then Theorem \ref{T:main}(a) applies: these  adjoint varieties are graded-FGH rigid at graded-order $4$.  Compare this with \cite[Prop. 5.2]{LRrigid}.  If $\fg=\fsl_{r+1}$, then Proposition \ref{T:filt} and \cite[Lemma 7.3]{LRrigid} imply the adjoint variety is graded-FGH rigid at graded-order $4$, if $r>2$; and at graded-order $5$, if $r=2$.
\end{example*}

\section{Proof of Theorem \ref{T:main}} \label{S:prf}

\subsection{The kernel $\fk$}

Given $\z =\z^j_k \sfv_j\ot\sfv^k\in \fgl(V)$, set 
$$
  \z_{\sfN\ot \sfL^*} \ := \ \z^\m_0 \, \sfv_\m \ot \sfv^0 \,.
$$

\begin{definition} \label{D:k}
Define 
$
  \fk  :=  \{ \z \in \fgl(V) \ | \ (\bu.\z)_{\sfN \ot \sfL^*} = 0 \ 
  \forall \ \bu \in \tSym^\bullet \fg_- \} \, .
$
\end{definition}

\noindent Since $\w$ is $\fg$--valued on $\cG$, it is an immediate consequence of Lemma \ref{L:key} that 
\begin{equation} \label{E:gink}
  \fg \ \subset \ \fk \, .
\end{equation}

\begin{lemma} \label{L:k=gmod}
The vector space $\fk$ is a $\fg$--module.
\end{lemma}

Given an algebra $\fa$, with universal enveloping algebra $U(\fa)$, we will make use of the Poincar\'e-Birkoff-Witt identification $U(\fa) = \tSym^\bullet\fa$.  

\begin{proof}
To see that $\fk$ is an $\fg_-$--module, let $A \in \fg_-$ and $\z\in\fk$.  Given $\bu \in U(\fg_-)$, write $\bu.(A.\z) = \bu'.\z$ with $\bu' = \bu A \in U(\fg_-)$.  Then $A.\z\in\fk$ follows from the definition of $\fk$.

To see that $\fk$ is an $\fg_{\ge0}$--module, let $E \in \fg_{\ge0}$.  As above, we have $\bu.(E.\z) = \be^s\bu_s.\z$ where $\bu_s \in U(\fg_-)$ and $\bx^s \in U(\fg_{\ge0})$.  The preimage of $\sfN\ot\sfL^*$ under $\be^s \in U(\fg_{\ge0})$ is contained in $\sfN\ot\sfL^*$.  By definition of $\fk$, we have $(\bu_s.\z)_{\sfN\ot\sfL^*} = 0$ for all $\mu$.  This yields $[ \bu. (E.\z)]_{\sfN\ot\sfL^*} = (\be^s\bu_s.\z)_{\sfN\ot\sfL^*}= 0$.  Thus $E.\z\in\fk$, and $\fk$ is a $\fg_{\ge0}$--module.
\end{proof}

\begin{lemma} \label{L:k=alg}
The kernel $\fk$ is a subalgebra of $\fgl(V)$.
\end{lemma}

\begin{proof}
Let $[\fk , \fk] = \tspan\{ [\z,\x] \ | \ \z,\x\in\fk \}$.  In order to show that $\fk$ is a subalgebra, we must show that $(\bu.[\fk,\fk])_{\sfN\ot\sfL^*} = 0$ for all $\bu \in \tSym^\bullet \fg_-$.  Let $u \in \fg_-$ and $\z,\x\in\fk$.  Then \begin{eqnarray*}
  u\,.\,[\z,\x] & = & \left[ u \,,\, [\z,\x] \right] 
  \ = \ \left[ \z \,,\, [\x,u] \right] \,+\, \left[ \x \,,\, [u,\z]\right] \, .
\end{eqnarray*}
By Lemma \ref{L:k=gmod} both $[\xi,u]$ and $[u,\z]$ lie in $\fk$.  Thus $u . [\fk,\fk] \subset [\fk,\fk]$.  Inductively, $\bu . [\fk , \fk] \subset [\fk,\fk]$ for all $\bu \in \tSym^\bullet \fg_-$.  So, to prove the lemma, it suffices to show that 
\begin{equation} \label{E:[k,k]}
  [\fk,\fk]_{\sfN\ot\sfL^*} \ = \ 0 \, .
\end{equation}

Given $\z \in \fgl(V)$, note that $\z_{\sfN\ot\sfL^*} = 0$ if and only if $\z(\sfL) \subset \sfL\op\sfT$.  Since $\sfL\op\sfT = \fg_{\le0} (\sfL)$, this is is equivalent to the existence of $\z'\in\fg_{\le0}$ such that $\z(\sfL) = \z'(\sfL)$.  Let $\z,\xi\in\fk$.  
\begin{eqnarray*}
  [\z,\x](\sfL) & = & \z\x (\sfL) \,-\, \x\z(\sfL)
  \ = \ \z( \x'\sfL) \,-\, \x( \z'\sfL) \\
  & = & [\z,\x'](\sfL) \,+\, \x'\z(\sfL)
  \,-\, [\x,\z'](\sfL) \,-\, \z'\x(\sfL) \, . 
\end{eqnarray*}  
By Lemma \ref{L:k=gmod}, $[\z,\x']$ and $[\x,\z']$ are elements of $\fk$.  Thus, $[\z,\x'](\sfL) - [\x,'\z](\sfL) \in \sfL\op\sfT$.  Also, 
$$
  \x'\z(\sfL) \,-\, \z'\x(\sfL) 
  \ = \ \x'\z'(\sfL) \,-\, \z'\x'(\sfL)
  \ = \ [\x',\z'](\sfL) \ \subset \ \sfL\op\sfT \, 
$$
since $[\x' \,,\, \z'] \in \fg_{\le0}$.  Thus $[\z,\x](\sfL) \subset \sfL\op\sfT$.  It follows that $[\fk , \fk] (\sfL) \subset \sfL \op \sfT$, establishing \eqref{E:[k,k]} and the lemma.
\end{proof}

Let $K \subset \tGL(V)$ be the connected subgroup with Lie algebra $\fk$.  Then the action of $K$ preserves $Z$.  Define $\cK := K\cdot \sfv$.  From Definition \ref{D:k} and Lemma \ref{L:key} we deduce

\begin{lemma} \label{L:cK}
$\cK$ is the maximal connected sub-bundle of $\cF_Z$, containing $\cG$, on which the coefficients of the Fubini forms are constant.
\end{lemma}

\noindent A corollary of Lemmas \ref{L:k=alg} \& \ref{L:cK} is that $K\supset G$ is the maximal connected subgroup of $\tGL(V)$ preserving $Z$.  Let $G_0 \subset K$ be a maximal semi-simple subgroup of $K$ containing $G$.  Then $G/P = G_0/P_0$ where $P_0$ is the parabolic subgroup of $G_0$ stabilizing the line $[\sfv_0] \in \bP V$.  
\begin{center}
\emph{Without loss of generality we assume that $G = G_0$.}
\end{center}

\begin{remark*}
By \cite{MR0039736, MR0150238}, when $G$ is simple, we have $G=G_0$ with the following exceptions:
\begin{circlist}
  \item $C_n \subset A_{2n-1}$ and $Z = \bC\bP^{2n-1}$;
  \item $G_2 \subset B_3$ and $Z \subset \bC\bP^6$ is the quadric hypersurface;
  \item $B_r \subset D_{r+1}$ and $Z \simeq B_r/P_r = D_{r+1}/Q_{r+1}$ is the Spinor variety.
\end{circlist}
From Table \ref{t:max_par}, we see that taking the larger group $G_0$, yields the smaller value for $\ttq$.
\end{remark*}
 
\subsection{Filtered systems} \label{S:filt_sys}

In order to prove Theorem \ref{T:main} we will realize the bundle $\wt \cF_{X}$ of Definition \ref{D:agree} as an integral manifold of a filtered Pfaffian system $\mathrm{Filt}_{p+2}(\fg_-\,,\,\fg^\perp_{\le p})$ whose rigidity (cf. Theorem \ref{T:filt}) is studied in \cite{LRrigid}.

\begin{definition} \label{D:filtsys}
Let $\Sigma$ be any manifold equipped with a filtration of its cotangent bundle 
$$
  \{0\} \,=\, T^*_a \,\subset\, T^*_{a+1} \,\subset\, T^*_{a+2}
  \,\subset\, \cdots \,\subset\, T^*_{b-1} \,\subset \,T^*_{b} 
  \,=\, T^*\Sigma \, ,
$$
$a < b \in \bZ$.  By convention $T^*_c = \{0\}$ for all $c \le a$, and $T^*_c = T^*\Sigma$ for all $c \ge b$.  Let $I \subset J \subset T^*\Sigma$ be sub-bundles of constant rank.  The filtration of $T^*\Sigma$ induces filtrations of $I$ and $J$.  Let $0<r\in\bZ$.  Integral manifolds $M$ of the \emph{$r$-filtered Pfaffian system} $\mathrm{Filt}_r(I,J)$ are the immersed submanifolds $i: M\to \Sigma$ such that 
$$
  i^*(I_c) \ \equiv \ 0 \quad \tmod \ i^*(J_{c-r}) \, ,
$$
for all $c$, and $i^*(J)=T^*M$ (the independence condition).  
\end{definition}

\begin{remark*}
The classical Pfaffian systems $(I,J)$ are given by taking the trivial filtration of $T^*\Sigma$ (that is, $b-a=1$) and $r=1$.
\end{remark*}

We now construct a filtered Pfaffian system on $\tGL(V)$ as follows.  Fix a $\fg$--module decomposition
$$
  \fgl(V) \ = \ \fg \,\op\, \fg^\perp \, .
$$
Let 
$$
  \fg \ = \ \op_{d=-\sfa}^\sfa \,\fg_d \,,\quad
  \fg^\perp \ = \ \op_{d=-\ttq}^\ttq \,\fg_d^\perp 
  \quad \hbox{ and } \quad
  \fgl(V) \ = \ \op_{d=-\ttq}^\ttq \,\fgl(V)_d 
$$
denote the $\ttE$--graded decompositions.  Define a filtration 
$$
  \{0\} \ \subset \ T^*_{-\sfq,\tId} \ \subset \ \cdots \ 
  \subset \ T^*_{\sfq,\tId} \ = \ T^*_\tId \tGL(V)
$$
of the cotangent space by 
$$
  T^*_{c,\tId} \ = \ \tAnn( \op_{d>c}\,\fgl(V)_d ) \, .
$$
Extend this to a filtration of the cotangent bundle $T^*\tGL(V)$ by setting $T^*_{c,A} = L_A{}^*(T^*_{c,\tId})$ for $A \in \tGL(V)$.  Given $p\in\bZ$, let $I \subset J \subset T^*\tGL(V)$ be the sub-bundles framed by $\w_{\fg^\perp_{\le p}}$ and $\w_{\fg^\perp_{\le p}} + \w_{\fg_{-}}$, respectively.  In a mild abuse of notation, let $\mathrm{Filt}_r(\fg_- \,,\, \fg^\perp_{\le p}) := \mathrm{Filt}_r(I,J)$ denote the corresponding filtered Pfaffian system on $\tGL(V)$.  

\begin{remark*} 
Any integral manifold of the filtered system $\mathrm{Filt}_r(\fg_- \,,\, \fg^\perp_{\le p})$ may be identified with an integral manifold of a linear Pfaffian system as follows.  Let $\fg^\perp \ot \fg_+ = \op_d \, (\fg^\perp \ot \fg_+)_d$ be the $\ttE$--graded decomposition.  Define
$$
  \Sigma_r \ = \ \tGL(V) \times (\fg^\perp \ot \fg_+)_{\ge r} \, .
$$
Let $\lambda$ be coordinates on $(\fg^\perp \ot \fg_+)_{\ge r}$, and write $\lambda = (\lambda_\ttq , \ldots , \lambda_{-\ttq})$, where $\lambda_s$ denotes the coordinates on $\fg^\perp_s \ot \fg_{\ge r-s}$.  Keeping in mind the identification $\fg_+ \simeq \fg_-^*$, define a linear Pfaffian system on $\Sigma_r$ defined by  
\begin{equation}\label{E:prolong}
  0 \ = \  \w_{\fg^\perp_s} - \lambda_s(\w_{\fg_-}) \,,
  \quad \hbox{for all } s \le p \,, 
\end{equation}
with independence condition
$$
  \Omega \ = \ \w^1_0 \wedge\cdots\wedge \w^m_0 \, .
$$
Any integral manifold $M$ of $\mathrm{Filt}_r(\fg_-\,,\,\fg^\perp_{\le p})$ lifts uniquely to an integral manifold \eqref{E:prolong}.  Conversely, any integral manifold of \eqref{E:prolong} is locally the lift of an integral manifold of the filtered system.
\end{remark*}  

\begin{definition} \label{D:filtrig}
The system $\mathrm{Filt}_r(\fg_- \,,\, \fg^\perp_{\le p})$ is \emph{rigid} if, for any connected integral manifold $M \subset \tGL(V)$, there exists $A \in \tGL(V)$ such that $L_A(M) \subset \cG$.  In particular, if $M$ is a sub-bundle of $\cF_X$ for some variety $X \subset \bP V$, then $A \cdot X \subset Z \simeq G/P$.
\end{definition}

The filtered system $\mathrm{Filt}_{p+2}(\fg_- \,,\, \fg^\perp_{\le p})$ is denoted by $(I^\sff_p \,,\, \Omega)$ in \cite{LRrigid}.\footnote{Be aware that the definition of an $r$-filtered Pfaffian system in this paper differs slightly from that in \cite{LRrigid}.  As a consequence $(I^\sff_p,\Omega)$ is a ``$(p+1)$--filtered system'' in \cite{LRrigid}, but a ``$(p+2)$--filtered system'' here.}

\begin{theorem}[{\cite[Theorems 1.7 \& 7.2]{LRrigid}}] \label{T:filt}
Let $G$ be a complex semi-simple Lie group and $G/P \hookrightarrow \bP V$ be a homogeneous embedding.  
\begin{a_list}
\item If no factor of $G/P$ corresponds to a quadric hypersurface or $A_r/P_\ttJ$, with $1$ or $r$ in $\ttJ$, then the system $\mathrm{Filt}_1(\fg_- \,,\, \fg^\perp_{\le -1})$ is rigid.
\item If no factor of $G/P$ corresponds to $A_r/P_\ttJ$, with $1$ or $r$ in $\ttJ$, then the system $\mathrm{Filt}_2(\fg_- \,,\, \fg^\perp_{\le 0})$ is rigid.
\end{a_list}
\end{theorem}

Theorem \ref{T:main} will follow from Theorem \ref{T:filt} and the following

\begin{proposition} \label{P:key}
Fix $p \ge -\ttq$, and set $d = \ttq+p+1$.  Let $X \subset \bP V$, and suppose that $X$ agrees with $G/P \inj \bP V$ to graded-order $d$.   Let $\wt\cF_X \subset \cF_X$ be the sub-bundle of Definition \ref{D:agree}.   Then $\wt\cF_X$ is an integral manifold of $\mathrm{Filt}_{p+2}(\fg_- \,,\, \fg^\perp_{\le p})$.  

In particular, if the filtered system is rigid, then $G/P \inj \bP V$ is graded-Fubini-Griffiths-Harris rigid at graded-order $\ttq+p+1$.
\end{proposition}

\begin{proof}
Let
$$
  \fgl(V) \ = \ \fgl(V)_\ttq \ \op \cdots \op \fgl(V)_{-\ttq} \,, \quad
  \ttq = \pi(\ttE) + \pi^*(\ttE) \, 
$$
be the $\ttE$--graded decomposition.  Let $\w_s$ denote the component of $\w$ taking value in $\fgl(V)_s$.  By Definition \ref{D:k}, the equation $\w_{\fg^\perp_s} = 0$ is equivalent to 
$$
   0 = (\bu\,.\,\w_s)_{\sfN\ot\sfL^*} \quad\hbox{for all}\quad
    \bu \in \tSym^\bullet \fg_-\, .
$$
If $\bu \in \tSym_{<-(s+\ttq)}\,\fg_-$, then $\bu \,.\, \w_s \in \fgl(V)_{<-\ttq} = \{ 0\}$; thus, the equation $0 =(\bu\,.\,\w_s)_{\sfN\ot\sfL^*}$ is trivial.  In particular, 
$$
  0 \ = \ ( \bu \,.\, \w)_{\sfN\ot\sfL^*} \ 
  \forall \ \bu \in \tSym_{\ge-(s+\ttq)} \, \fg_-\quad \Longrightarrow \quad
  \w_{\fg^\perp_s} \ = \ 0 \, .
$$
It follows from Lemma \ref{L:key} and inspection of \eqref{E:F} that, in order to impose $\w_{\fg^\perp_s} = 0$, it suffices to specify the Fubini forms up to graded-order $s+\ttq+\tta$, cf. \eqref{E:Zg}.

Upon closer inspection of \eqref{E:F}, we see that specifying the Fubini forms up to graded-order $s+\ttq+\ell$, with $1 \le \ell \le \tta$, yields $\w_{\fg^\perp_s} = \varsigma(\w_{\fg_{-(\ell+1)}} + \cdots + \w_{\fg_{-\tta}})$ for functions $\varsigma$ taking value in $\fg^\perp_s \ot (\fg_{\ell+1} \op \cdots \op \fg_\tta)$.  Comparing with \eqref{E:prolong}, we see that $\wt\cF_X$ is an integral manifold of the system $\mathrm{Filt}_{p+2}(\fg_- \,,\, \fg^\perp_{\le p})$.
\end{proof}

\section{An invariant description of the Fubini forms for \boldmath $G/P \hookrightarrow \bP V$ \unboldmath }

In this section we give the invariant description of the Fubini forms that was promised in Section \ref{S:I4}.

\subsection{Fubini forms}\label{S:F5}

Given $x \in \bP V$,  let $\hat x \subset V$ denote the corresponding line.  Let $X\subset \BP  V$ be a projective variety and let $x\in X$ be a smooth point.   The line bundle with fibre $L_x = \hat x$ over $x \in X$ is $L=\cO_{X}(-1)$.     The quotient map $ V^*\to V^*/\tAnn(\hat x) = \cO_{\bP  V}(1)_x$ gives rise to a sequence of maps defined as follows:   Let $\fm_x \subset \cO_{\bP  V, x}$ denote the maximal ideal of functions vanishing at $x$ to first order.  Recall that $T^*_xX:=\fm_x/\fm_x^2$.  More generally, $\tSym^kT^*_xX=\fm_x^k/\fm_x^{k+1}$.  

Let $\BF_x^0: V^* \to \cO_{X,x}(1)/\fm_x(1)$ and $\BF_x^1: \tker (\BF_x^0) \to \fm_x(1)/\fm_x^2(1)$ denote the natural projection maps.  Define $\BF^k_x : \tker(\BF_x^{k-1}) \to \fm_x^k(1)/\fm^{k+1}_x(1)$ inductively.  Set $N^*_{k,x}X(1)=\tker (\BF^{k-1}_x)$.  The \emph{$k$-th fundamental form} (twisted by $\cO_X(1)$) is the induced map
$$
\BF^k_x: N^*_{k,x}X(1)\to \tSym^kT^*_xX(1).
$$
The fundamental forms are well defined tensors at each smooth point of $X$, and are a subset of the Fubini forms.  They describe how $X$ is infinitesimally moving away from its $(k-1)$-st osculating space $\wh T\up{k-1}_xX:=(N^*_k(1))^\perp\subset V$ at order $k$.  

Recall the osculating sequence \eqref{oscseq}.  Let $(x^0,x^{\mu_1},x^{\mu_2}  \hd x^{\mu_f})$ be local coordinates on $\bP V$ adapted to the osculating filtration at a general point $x$.  The $x^{\mu_1}$ represent tangential coordinates, and we distinguish them by replacing the index $\mu_1$ with $\a$.  Locally $X$ is a graph over its tangent space at $x$.  The $x^{\mu_j}(x^{\a})$ have a Taylor series beginning at order $j$, and the coefficients of the $j$-th fundamental form are the leading terms of the Taylor series.  The other terms in the Taylor series contain geometric information but do not lead to such clearly defined tensors on $X$.  Rather they lead to the Fubini forms $F^k\in \tSym^kT^*_xX\ot N_xX$, a sequence of relative differential invariants which are defined on the bundle $\cF_X$ of first-order adapted frames.  

\begin{remark*}
It is possible to define a tensor on $X$ that contains essentially the same information of the $k$-th Fubini form.  One uses the $k$-th fundamental form of the $(k-1)$-st Veronese re-embedding of $X$.  See \cite{MR1420349} for details. 
\end{remark*}

\subsection{Fubini forms of \boldmath $G/P \hookrightarrow \bP V$ \unboldmath}

The universal enveloping algebra $U(\fg)= \fg^{\ot}/\{ \x\ot\z- \z\ot \x - [\x,\z]\}$ inherits a filtration from the tensor algebra.  Let $U(\fg)^k$ denote the $k$-th term in this filtration.  

The fundamental forms may be described by a commutative diagram as follows.   
\begin{circlist}
\item The osculating spaces of $Z$ at $o=[\sfv_0]$ are $\wh T\up k_oZ=U(\fg)^k.\sfv_0$.  Complete $\sfv_0$ to a basis $\sfv = (\sfv_0 , \ldots, \sfv_n)$ of $V$.  Let $(\sfv^0,\ldots,\sfv^n)$ denote the dual basis.  Define $U(\fg) \to \wh T^{(k)}_oZ \ot L_o^*$ by $\bu \mapsto (\bu.\sfv_0)\ot\sfv^0$.
\item The normal space of $Z$ at $o$ is $N_oZ = (V / \wh T_oZ)\ot L_o^*$, and the $k$-th normal space is $N^k_oZ = (\wh T^{(k)}_o Z \,/\, \wh T^{(k-1)}_oZ)\ot L_o^*$.  Let $T^{(k)}_oZ \ot L_o^* \to N^k_oZ$ be the induced quotient map.
\item The Poincar\'e-Birkoff-Witt theorem and the identification $\fg\to \fg/\fp =T_oZ$ define a map $U(\fg)^k \to \tSym^k(T_oZ)$. 
\end{circlist}
The resulting diagram below is commutative and characterizes $\BF^k$, cf. \cite[Proposition 2.3]{MR1966752}.
\begin{center} \setlength{\unitlength}{5pt}
\begin{picture}(40,12)
\put(-1,0){$\tSym^k(T_oZ) $}
\put(4,9){$U(\fg)^k$} 
\put(23,9){$(\wh T^{(k)}_o Z)\ot L_o^*$} 
\put(23,0){$N^k_oZ$}
\put(12,10){\vector(1,0){9}}
\put(12,0.5){\vector(1,0){9}} \put(16,1.5){$\BF^k$}
\put(7,7){\vector(0,-1){3}} \put(26,7){\vector(0,-1){3}}
\end{picture} \end{center}

This observation may be generalized to the Fubini forms of $Z$ as follows.  Assume the frame $\sfv = (\sfv_0,\ldots,\sfv_n) \in \cF_Z$ over $o = [\sfv_0]$ selected above respects the splitting \eqref{E:split0}.     
Define a map
\begin{equation} \label{ugmap}
\renewcommand{\arraystretch}{1.3}
\begin{array}{rcl}
  U(\fg)^{k-1} & \to & (V\ot V^*)\,\ot\,(V\ot V^*) \\
  \hbox{by} \quad \bu & \mapsto & \left(\bu.(\sfv_i\ot \sfv^j)\right) 
  \,\ot\, (\sfv_j\ot \sfv^i) \, .
\end{array}
\end{equation}
Here and throughout, repeated indices are to be summed over.  Because
$(\sfv_i\ot \sfv^j)\ot(\sfv_j\ot \sfv^i) = \tId\in \tEnd(\tEnd(V))$, the map \eqref{ugmap} is independent of the choice of basis $\sfv$.  The splitting \eqref{E:split0} induces a projection 
$$
  (V\ot V^*) \,\ot\, (V\ot V^*) \ \to \ (\sfN\ot \sfL^*) \,\ot\, (\sfL\ot \sfT^*)
   \,.  
$$
As an element of $N_oZ \ot \tSym^{k+1}(T^*_oZ)$, the $(k+1)$--st Fubini form may be viewed as a map $\tSym^k T_oZ \to (N_oZ)\ot (T_o^*Z)$.  Finally, making use of \eqref{E:sfId}, the Fubini form may be considered as a map $F^{k+1}_{\sfv} : \tSym^k(\sfT\ot\sfL^*) \to (\sfN\ot\sfL^*) \ot (\sfL\ot\sfT^*)$.

\begin{proposition}\label{invarfubgp} 
Let $G/P\subset \bP  V$ be a homogeneously embedded rational homogeneous variety and let $o=[\tId]\in G/P$.  Given a choice $\fh \subset \fb \subset \fg$ of Cartan and Borel subalgebras of $\fg$, let $V=\sfL\op \sfT\op \sfN$ be the induced splitting \eqref{E:split0}, and $\sfv \in \cF_Z$ a basis of $V$ respecting the splitting.  The diagram below commutes and characterizes the $(k+1)$--st Fubini form at $\sfv \in \cF_Z$.
\begin{center} \setlength{\unitlength}{5pt}
\begin{picture}(45,12)
\put(-2.5,0){$\tSym^{k}(\sfT\ot \sfL^*)$}
\put(4,9){$U(\fg)^{k}$} 
\put(25,9){$(V\ot V^*)\,\ot\,(V\ot V^*)$} 
\put(26,0){$(\sfN\ot \sfL^*)  \,\ot\, (\sfL \ot \sfT^*)$}
\put(14,10){\vector(1,0){9}}
\put(14,0.5){\vector(1,0){9}} \put(16.5,1.6){$F^{k+1}_\sfv$}
\put(7,7){\vector(0,-1){3}} \put(36.5,7){\vector(0,-1){3}}
\end{picture} \end{center}
\end{proposition}

\begin{proof}
Corollary \ref{C:deg} yields a substantial simplification of \eqref{E:F} on $\cG$.  Write 
$$
  \w \ = \ \eta \,+\, \w_\fp \, ,
$$
where $\eta$ and $\w_{\ge0}$ are respectively the $\fg_-$ and $\fp$--valued components of the $\fg$--valued Maurer-Cartan form.  Note that $\w^\a_0 = \eta^\a_0$.  It follows that the $\w_\fp$ terms on the left-hand side of \eqref{E:F} must all cancel.  This, along with Corollaries \ref{C:cnst} \& \ref{C:deg}, yields
\begin{equation} \label{E:F_homog}
  r^\m_{\a_1 \ldots \a_k \b} \, \eta^\b_0 \ = \
  - \, r^\n_{\a_1 \ldots \a_k} \, \eta^\m_\n \hfill 
  \ + \ k \, r^\m_{\b(\a_2 \ldots \a_k} \, \eta^\b_{\a_1)} 
  \quad \hbox{ on } \ \cG \,.
\end{equation}
This equation yields the inductive formula \eqref{E:Find} below as follows.

We assume that $\sfv$ respects the decomposition \eqref{E:split0}.  Let $\sfv^* =(\sfv^0,\ldots,\sfv^n)$ denote the basis of $V^*$ dual to $\sfv$.  Then $\{\ul{\sfv}_\a := \sfv_\a \ot \sfv^0\}_{\a=1}^m$ spans $\sfT\ot\sfL^* \simeq T_oZ$, and $\{\ul{\sfv}_\m := \sfv_\m \ot \sfv^0\}_{\m=m+1}^n$ spans $\sfN\ot\sfL^* \simeq N_oZ$.  Let $\{\ul{\sfv}^\a := \sfv_0 \ot \sfv^\a \}_{\a=1}^m$ denote the dual basis of $\sfL \ot \sfT^*$, and set $\ul{\sfv}^{\a_1\cdots\a_k} = \ul{\sfv}^{\a_1} \circ \cdots \circ \ul{\sfv}^{\a_k} \in \tSym^k(\sfL \ot \sfT^*)$.  Then
$$
  F^k_\sfv \ = \ r^\m_{\a_1\cdots\a_k} \, \ul{\sfv}_\m \ot \ul{\sfv}^{\a_1\cdots\a_k} 
  \ \in \ (\sfN\ot\sfL^*) \ot \tSym^k(\sfL\ot\sfT^*) \, .
$$

Observe that the terms appearing on the right-hand side of \eqref{E:F_homog}, at $\sfv \in \cG$, are precisely the coefficients of $\eta.F^k_\sfv$.  Applying the operation $\tilde u_\b \lefthook$ to \eqref{E:F_homog} yields
\begin{equation} \label{E:Find}
  F^{k+1}_\sfv \ = \ (u_\b . F^k_\sfv) \circ \ul{\sfv}^\b \,. 
\end{equation}
Equation \eqref{E:r2} asserts that  
\begin{equation} \label{E:base}
  F^2_\sfv \ = \ \left( (u_\a u_\b) . \sfv_0 \right)^\m \, \ul{\sfv}_\m \,\ot\, 
  \ul{\sfv}^{\a\b} \, ,
\end{equation}
where $\left( (u_\a u_\b) . \sfv_0 \right)^\m$ denotes the $\sfv_\m$ coefficient of $(u_\a u_\b) . \sfv_0\in V$ with respect to the basis $\sfv$.  An inductive argument with \eqref{E:base} and \eqref{E:Find} yields Proposition \ref{invarfubgp}.
\end{proof}

\bibliography{refs} 
\bibliographystyle{plain}
\end{document}